\documentclass[12pt, reqno]{amsart}
\usepackage{amsmath, amsthm, amscd, amsfonts, amssymb, graphicx, color}
\usepackage[bookmarksnumbered, colorlinks, plainpages]{hyperref}
\hypersetup{colorlinks=true,linkcolor=red, anchorcolor=green, citecolor=cyan, urlcolor=red, filecolor=magenta, pdftoolbar=true}

\usepackage[scr]{rsfso}
\let\div\undefined
\DeclareMathOperator{\div}{div}

\newtheorem{theorem}{Theorem}[section]
\newtheorem{lemma}[theorem]{Lemma}
\newtheorem{proposition}[theorem]{Proposition}
\newtheorem{corollary}[theorem]{Corollary}
\theoremstyle{definition}
\newtheorem{definition}[theorem]{Definition}

\theoremstyle{remark}

\numberwithin{equation}{section}

\begin{document}

\setcounter{page}{1}

\title[Sobolev Trace Spaces and Dirichlet Problem]{ Variable Exponent Sobolev Trace Spaces and Dirichlet Problem in Axiomatic Nonlinear Potential Theory}

\author[MOHAMED BERGHOUT
]{MOHAMED BERGHOUT\\
	Ibn Tofail University –Kenitra – Morocco\\
	B.P.242-Kenitra 14000 ESEF.\\
	Mohamed.berghout@uit.ac.ma\\
moh.berghout@gmail.com}



\dedicatory{\textbf {This paper is dedicated to my mother with deep estimate and  love} }



\date{Received: xxxxxx; Revised: yyyyyy; Accepted: zzzzzz.
\newline \indent $^{*}$Corresponding author}

\begin{abstract}
We give a news characterization of variable exponent Sobolev trace spaces. We construct The Perron-Weiner-Brelot operator in nonlinear harmonic space and
we give sufficient condition for which this operator is injective. 	 
\end{abstract} 
\maketitle
\section{\textbf {Introduction}}
This paper has dual goals. One goal is to use potential analysis to give a news characterization of variable exponent trace spaces. A second goal is to discuss the axiomatic nonlinear potential theory associated with the perturbed $p(.)$-Laplacian operator
\begin{equation*}
	\mathbf{\mathscr{L}_{p(.)}}u:=-\Delta_{p(.)}u+\mathscr{B}(.,u),	
\end{equation*}
where $p$ is a measurable function,  $\Delta_{p(.)}:=\div\left( \left| \nabla u \right|^{p(.)-2} \nabla u\right)$ is the $p(.)$-Laplacian operator and $\mathscr{B}$ is a given Carath\'{e}odory
functions satisfies some structural conditions. In the present introduction we motivate the questions that we address and we state the main results.

One of the most important motivations for the theory of Sobolev spaces with variable exponent comes from nonlinear potential theory. A typical task of nonlinear potential theory in variable exponent Sobolev spaces is to find a $p(.)$-harmonic function $ h $ on a bounded open set $ \Omega\subset \mathbb{R}^{n}$ which continuously extends the given boundary data $ f \in \mathscr{C}(\partial \Omega) $ so that the $ p(.) $-Laplace equation $\Delta_{p(.)}u=0 $
is satisfied on $ \Omega $  and $ u = f  $ on $ \partial \Omega $. This is the $ p(.)$-Dirichlet problem. Its reformulation in terms of Sobolev spaces is to extend $ f\in \mathscr{C}(\partial \Omega) $ to $ u\in \mathscr{C}(\overline{\Omega}) \cap \mathcal{W}^{1,p(.)}(\Omega) $, where $ \mathcal{W}^{1,p(.)}(\Omega) $ is the space of measurable functions $ u : \Omega\longrightarrow \mathbb{R}$ such that the modular 
$\rho_{1,p(.)}^{\Omega}(u):=\displaystyle\int_{\Omega}\left(\left|u(x) \right|^{p(x)}+\left|\nabla u(x) \right|^{p(x)}  \right)\ dx$ is finite, with $\nabla u$ is the distributional gradient. Already the theory of Sobolev spaces, one of the major themes of this paper, can serve as some kind of axiomatic approach to nonlinear potential theory, we refer to \cite{Heinonen}.   

Functions spaces with variable exponent have been intensely investigated in the recent years. One of such spaces is the Lebesgue and Sobolev spaces with variable exponent.
They were introduced by W. Orlicz in 1931 \cite{Orlicz}; their properties were further developed by H. Nakano as special cases of the theory of modular spaces \cite{Nakano}. In the ensuing decades they were primarily considered as important examples of modular spaces or the class of Musielak-Orlicz spaces. In the beginning these spaces had theoretical interest. Later, at the end of the last century, their first use beyond the function spaces theory itself, was in variational problems and studies of $p(.)$-Laplacian operator, which in its turn gave an essential impulse for the development of this theory. For more details on these spaces,  see  \cite{Hasto,kokivalish1,kokivalish2,FAN,kovacik}.

We now give the main results of the paper. First, we introduce some notations which will be observed in this paper. Throughout this paper we will use the following notations: $\mathbb{R}^{n}$ is the $n$-dimens\-ional Euclidean space, and $n\in \mathbb{N}$ always stands for the dimension of the space. $\Omega \subset \mathbb{R}^{n}$ is a open set equipped with the $n$-dimensional Lebesgue measure. For constants we use the letter $C$ whose value may change even within a string of estimates. The ball with radius $ r $ and center $ x\in \mathbb{R}^{n}$  will be denoted by 
$ B(x,r) $. The closure of a set $A$ is denoted by $\overline{A}$ and the  topological boundary of $ A $ is denoted by $ \partial A $. The complement of $ A $ will be denoted by $ A^{c}$. We use the usual convention of identifying two $\mu$-measurable function on $A$ (a.e. in $A$, for short) if they agree almost everywhere, i.e. if they agree up to a set of $\mu$-measure zero. The characteristic function of a set $ E \subset A $ will be denoted by $ \chi_{E} $.  The Lebesgue integral of a Lebesgue measurable function $f:\Omega\longrightarrow \mathbb{R}$, is defined in the standard way and denoted by $\displaystyle\int_{\Omega} f(x)\ dx$. We use the symbol $ := $ to define the left-hand side by the right-hand side. For measurable functions $u,v: \Omega\longrightarrow \mathbb{R}$, we set  $u^{+}:=\max \left\lbrace u,0\right\rbrace $, 
$u^{-}:=\max \left\lbrace -u,0\right\rbrace$, $ u\vee v:=\max\left\lbrace u,v \right\rbrace$ and $u\wedge v:=\min\left\lbrace u,v \right\rbrace$ . We denote by $L^{0}(\Omega)$ the space of all $\mathbb{R}$-valued measurable functions on $\Omega$. We denote by $\mathscr{C}(\Omega)$ the space of continuous functions on $\Omega$. By $\mathscr{C}_{c}(\overline{\Omega})$ we design the space of continuous functions on $\overline{\Omega}$ with compact support in $\overline{\Omega}$. We denote by $\mathscr{C}(\overline{\Omega})$ the space of uniformly continuous functions equipped with the supremum norm $\|f\|_{\infty}=\sup_{x\in\overline{\Omega}}|f(x)|$. By $\mathscr{C}^{k}(\overline{\Omega})$, $k\in \mathbb{N}$, we denote the space of all function $f$, such that $\partial_{\alpha} f :=\dfrac{\partial^{\left| \alpha \right|}f}{\partial^{\alpha_{1}}x_{1}........\partial^{\alpha_{n}}x_{n}} \in \mathscr{C}(\overline{\Omega})$ for all multi-index $\alpha = \left( \alpha_{1},\alpha_{2},...,\alpha_{n} \right)$,  $\left| \alpha \right|:=\alpha_{1}+\alpha_{2}+......+\alpha_{n} \leq k$. The space is equipped with the norm $\sup _{\left| \alpha \right|\leq k} {\left\| \partial_{\alpha} f\right\|}_{\infty}$, $\mathscr{C}^{\infty}(\overline{\Omega})=\bigcap_{k}\mathscr{C}^{k}(\overline{\Omega})$. The set of smooth functions in $\Omega$ is denoted by $\mathscr{C}^{\infty}(\Omega)$ - it consists of functions in $\Omega$ which are continuously differentiable arbitrarily many times. The set $\mathscr{C}^{\infty}_{0}(\Omega)$ is the subset of $\mathscr{C}^{\infty}(\Omega)$ of functions which have compact support.

Next, we introduce variable exponent Lebesgue and Sobolev spaces on open sets $ \Omega $ of $\mathbb{R}^{n}$. Let $p:  \Omega \longrightarrow \left[ 1,\infty \right)$ be a measurable function (called the variable exponent on $\Omega$). $ \mathcal{P}(\Omega)$ is already used as a set of variable exponent on $\Omega$. We set $p^{+}=\text{esssup}_{x\in \Omega} \ p(x)$ and $p^{-}=\text{essinf}_{x\in \Omega} \ p(x)$ and throughout this paper we assume that 
\begin{eqnarray*}
1<p^{-}\leq p(x)\leq p^{+} <\infty. 
\end{eqnarray*}
Notice that by \cite[ Proposition 4.1.7]{Hasto}, we can extend $p$ to all of $\mathbb{R}^{n}$.\\ 
 The variable exponent Lebesgue space $ L^{p(.)}(\Omega) $ is the family of the equivalence classes of functions defined by
\begin{equation*}
L^{p(.)}(\Omega):=\left\lbrace u\in L^{0}(\Omega):\rho_{p(.)}(\lambda u)=\int_{\Omega}\left|\lambda u(x)\right|^{p(x)} \ dx<\infty, \  
\text{for some} \ \lambda>0\right\rbrace. 
\end{equation*}
The function $
\rho _{p(.)}: L^{p(.) }(\Omega)\longrightarrow
\left[ 0,\infty \right)$ is called the modular of the space 
$L^{p(.)}(\Omega)$. We define a norm, the so-called Luxembourg norm, in this space by 
\begin{equation*}
\left\| u\right\|_{L^{p(.) }}=\inf \left\lbrace  \lambda >0:\rho
_{p(.) }\left( \frac{u}{\lambda }\right) \leq 1\right\rbrace.
\end{equation*}
As in the classical case, the dual variable exponent function $
p^{'}$ of $p$ is given by $\frac{1}{p(x) }+\frac{1}{{p^{'}(x)}}=1$ and dual space for $L^{p(.)}(\Omega) $ is $L^{p^{'}(.)}(\Omega)$.  
If $v\in
L^{p^{'}(.)}(\Omega)$ and $u\in L^{p(.)}(\Omega) $ then the following
H\"{o}lder's inequality holds:
\begin{equation*}
	\int_{\Omega } \left|uv\right| \ dx \leq 2\left\|   u\right\|_{L^{p(.) }(\Omega) }\left\| v\right\|_{L^{p^{'}(.)}(\Omega) }.
\end{equation*}
We define the variable exponent Sobolev space $ \mathcal{W}^{1,p(.)}(\Omega) $ as follows:
\begin{equation*}
\mathcal{W}^{1,p(.)}(\Omega):=\left\lbrace u\in L^{p(.)}(\Omega), \left| \nabla u \right|\in L^{p(.)}(\Omega)\right\rbrace. 
\end{equation*}
The space $ \mathcal{W}^{1,p(.)}(\Omega) $ is a Banach space endowed with the norm 
\begin{equation*}
\left\|  u\right\| _{ \mathcal{W}^{1,p(.)}(\Omega)}=\left\|  u\right\| _{ L^{p(.) }(\Omega)}+\left\| \nabla u\right\|_{  L^{p(.) }(\Omega)}.
\end{equation*}
We define a modular on $\mathcal{W}^{1,p(.)}(\Omega)$ by 
\begin{eqnarray*}
	\rho^{\Omega}_{1,p(.)}(u):=\int_{\Omega}\left| u(x)\right|^{p(x)}\ dx+\int_{\Omega}\left|\nabla u(x)\right|^{p(x)}\ dx,	
\end{eqnarray*}
when $\Omega=\mathbb{R}^{n}$ we denote it by 
$\rho_{1,p(.)}$. The definition of the space 
$L^{p(.)}(\mathbb{R}^{n})$ and $\mathcal{W}^{1,p(.)}(\mathbb{R}^{n})$ is analogous to $L^{p(.)}(\Omega)$ and $\mathcal{W}^{1,p(.)}(\Omega)$; one just changes every occurrence of $\Omega$ by $\mathbb{R}^{n}$.\\ 
The Sobolev space $ \mathcal{W}^{1,p(.)}_{0}(\Omega)$
with zero boundary values is the closure of the set of $ \mathcal{W}^{1,p(.)}(\Omega)$-functions with compact support, i.e.
\begin{equation*}
\left\lbrace u\in \mathcal{W}^{1,p(.)}(\Omega):u=u{\chi_{K}} \ \text{for a compact} \ K\subset\Omega  \right\rbrace 
\end{equation*}
in $ \mathcal{W}^{1,p(.)}(\Omega) $. When smooth functions are dense, we can also use the closure of $\mathscr{C}^{\infty}_{0}(\Omega)$ in $ \mathcal{W}^{1,p(.)}(\Omega)$. The space $ \mathcal{W}^{1,p(.)}_{0}(\Omega)$ is a closed Banach subspace of $ \mathcal{W}^{1,p(.)}(\Omega)$ which is separable if $  p^{+}<\infty $, and reflexive and uniformly convex if $ 1<p^{-}\leqslant  p^{+} <\infty$, we refer to \cite{Hasto3}.

A subspace $I$ of $\mathcal{W}^{1,p(.)}(\Omega)$ is called an ideal if for 
$u\in I,\ v\in \mathcal{W}^{1,p(.)}(\Omega)$, $ \left|v \right| \leqslant \left|u \right|$ a.e. implies that $v\in I$. The closed lattice ideals of the Sobolev spaces $\mathcal{W}^{1,p(.)}(\Omega)$ are those subspaces which consist of all functions which vanish on a prescribed set. To be precise, we have the following result:
\begin{theorem}
	Let $ p\in \mathcal{P}(\Omega) $ with $ 1<p^{-}\leqslant p^{+}<\infty$. Then the space $\mathcal{W}^{1,p(.)}_{0}(\Omega)$ is a closed ideal in $\mathcal{W}^{1,p(.)}(\Omega)$. Moreover, there exists a borel set 
	$\mathcal{B}$ such that $$\mathcal{W}^{1,p(.)}_{0}(\Omega)=\mathcal{W}^{1,p(.)}_{0}(\mathcal{B})$$. 
\end{theorem}
Let $p\in \mathcal{P}(\Omega)$. For a set $ E\subset \mathbb{R}^{n}$, we denote
\begin{eqnarray*}	
	\mathcal{A}_{p(.)}(E):=\left\lbrace u\in \mathcal{W}^{1,p(.)}(\mathbb{R}^{n}): u\geq 1 \ \textnormal{a.e. on a neighbourhood of } E \right\rbrace.    	
\end{eqnarray*}
The Sobolev $p(.)$-capacity of the set $E$ is the number defined by
\begin{eqnarray*}
	C_{p(.)}(E):=\inf_{u\in\mathcal{A}_{p(.)}(E)}\rho_{1,p(.)}(u).	   	
\end{eqnarray*}
In case $\mathcal{A}_{p(.)}(E)=\varnothing$, we set $C_{p(.)}(E)=\infty$. 

Now, we introduce an alternative to the Sobolev $p(.)$-capacity, in which the capacity of a set is taken relative to a open subset. For $ p\in \mathcal{P}(\Omega)$, we let
\begin{eqnarray*}
	\tilde{\mathcal{W}}^{1,p(.)}(\Omega):= \overline{{\mathcal{W}}^{1,p(.)}(\Omega)\cap \mathscr{C}_{c}(\overline{\Omega})}^{{\mathcal{W}}^{1,p(.)}(\Omega)}
\end{eqnarray*}
and
\begin{eqnarray*}
	\mathcal{R}^{ \overline{\Omega}}_{p(.)}(O):=\left\lbrace u\in \tilde{{\mathcal{W}}}^{1,p(.)}(\Omega): u\geq 1 \ \textnormal{ a.e. on}\ O \right\rbrace.    	
\end{eqnarray*}
\begin{definition}
Let $O\subset \overline{\Omega}$ be a relatively open set, that is, open with respect to the relative topology of $\overline{\Omega}$. 
We define the relative $p(.)$-capacity of $O$, with respect to $\Omega$, by
\begin{eqnarray*}
	C_{p(.)}^{\overline{\Omega}}(O):=\inf_{u\in\mathcal{R}^{ \overline{\Omega}}_{p(.)}(O)}{\rho}_{1,p(.)}^{\Omega}(u).	   	
\end{eqnarray*}
For any set $ E\subset \overline{\Omega}$,
\begin{eqnarray*} 
	C_{p(.)}^{\overline{\Omega}}(E)=\inf\left\lbrace C_{p(.)}^{\overline{\Omega}}(O): O \ \textnormal{relatively open in}\ \overline{\Omega} \ \textnormal{containing}\ E  \right\rbrace. 
\end{eqnarray*}
\end{definition}
\begin{definition}
	\item A set $P\subset\Omega$ is called $p(.)$-relatively polar if $C_{p(.)}^{\overline{\Omega}}(P)=0$.
	\item We say that a property holds on a  set $ A\subset\Omega$ $ p(.) $-relatively quasieverywhere ($ p(.) $-r.q.e., for short) if there exists a $p(.)$-relatively polar set $P\subset A$ such that the property holds everywhere on $ A\setminus P$.
	\item  A function $ u:\Omega\longrightarrow \mathbb{R}$ is said to be $p(.)$-relatively quasicontinuous ($ p(.) $-r.q.c., for short) if for every $ \varepsilon>0 $, there exists a relatively open set $ O_{\varepsilon}\subset\Omega $ such that $C_{p(.)}^{\Omega}(O_{\varepsilon})<\varepsilon$ and $ u $ is continuous on $ \Omega\setminus O_{\varepsilon}$.    
\end{definition}
In the following, we characterize $\mathcal{W}^{1,p(.)}_{0}(\Omega)$ in terms of the 
$p(.)$-relative capacity. 
\begin{theorem}
	Let $\Omega\subset \mathbb{R}^{n}$ be an open set. Then
	\begin{eqnarray*}
		\mathcal{W}^{1,p(.)}_{0}(\Omega)=\left\lbrace u\in \tilde{\mathcal{W}}^{1,p(.)}(\Omega):\ \tilde{u}=0 \ p(.)\textnormal{-r.q.e. on} \ \partial\Omega   \right\rbrace, 	
	\end{eqnarray*} 
	where $\tilde{u}$ denotes the relatively $p(.)$-quasicontinuous representative of $u$.
\end{theorem}
\begin{corollary}
	Let $\Omega\subset \mathbb{R}^{n}$ be an open set. Then the following assertions are equivalent.
	\begin{enumerate}
		\item $ C_{p(.)}^{\overline{\Omega} }(\partial\Omega)=0$;
		\item $\mathcal{W}^{1,p(.)} _{0}(\Omega)=\tilde{\mathcal{W}}^{1,p(.)}(\Omega)$.
	\end{enumerate}
\end{corollary}
The next theorem present another description of the space $\mathcal{W}^{1,p(.)} _{0}(\Omega)$.
\begin{theorem}
	Let $\Omega\subset \mathbb{R}^{n}$ be an open set. Suppose that 
	$F$ is a closed subset of $\mathbb{R}^{n}$ such that $\Omega=\mathbb{R}^{n}\setminus F$. If $ F $ is of $p(.) $-Sobolev capacity zero, then  
	\begin{eqnarray*}
		\mathcal{W}^{1,p(.)}_{0}(\Omega)=\left\lbrace u_{\arrowvert\Omega}: \  u\in \mathcal{W}^{1,p(.)}(\mathbb{R}^{n})\right\rbrace. 	
	\end{eqnarray*} 
\end{theorem}
\begin{definition}
	Let $\Omega\subset \mathbb{R}^{n}$ be an open set. We say that $\Omega$ is $p(.)$-regular in capacity if $ C_{p(.)}(B(x,r)\setminus\Omega)>0$ for all $ x\in\partial\Omega $.	
\end{definition}	
\begin{definition}
	We say that an open set $\Omega\subset \mathbb{R}^{n}$ has the $p(.)$-zero property if $u\in \mathscr{C}(\partial\Omega)$ and $u=0\ p(.)$-q.e on $\partial\Omega$ implies that $u(x)=0$ for every 
	$x\in \partial\Omega$.
\end{definition}
The next theorem, says that our notions of $p(.)$-regularity in capacity and $p(.)$-zero property are good behavior under $p(.)$-capacities.

\begin{theorem}
	Let $ \Omega\subset \mathbb{R}^{n}$ be an open set. Then the following assertions are equivalent.
	\begin{enumerate}
		\item $\Omega$ has the $p(.)$-zero property;
		\item  Every function $u\in \mathcal{W}^{1,p(.)}_{0}(\Omega)\cap \mathscr{C}(\overline{\Omega})$ is zero everywhere on $\partial\Omega$;
		\item $\Omega$ is $p(.)$-regular in capacity;
		\item $ C_{p(.)}(B(x,r)\cap\partial\Omega)>0$ for all 
		$x\in \partial\Omega$;
		\item $ C_{p(.)}^{\overline{\Omega}}(B(x,r)\cap\partial\Omega)>0$ for all 
		$x\in \partial\Omega$;
		\item $u\in \mathscr{C}(\partial\Omega)$ and $u=0$ $p(.)$-r.q.e. on $ \partial\Omega $ implies that 
		$u(x)=0$ for every $x\in \partial\Omega$.  	 
	\end{enumerate}	 
\end{theorem}

 For 
$f\in \mathscr{C}(\partial\Omega)$, we consider the Dirichlet problem:
\begin{equation*}	
\left\{
\begin{array}{rl}	
\mathbf{\mathscr{L}_{p(.)}}u:=-\Delta_{p(.)}u+\mathscr{B}(.,u)=0 & \text{in}\ \Omega; \\
u=f & \text{on}\ \partial\Omega. 
\end{array}
\right.
\end{equation*}
 We define the  
$\mathcal{H}_{\mathscr{L}_{p(.)}}$-harmonic sheaf as follows:
\begin{equation*}
\mathcal{H}_{\mathscr{L}_{p(.)}}(\Omega)=
\left\lbrace u\in \mathscr{C}(\Omega):\mathscr{L}_{p(.)}u=0 \right\rbrace. 
\end{equation*}
We construct the Perron-Wiener-Brelot operator
\begin{equation*}
f\mapsto {\mathscr{L}^{\Omega}_{p(.),f}} 	
\end{equation*}
from a set of real extended functions $ f : \partial\Omega \longrightarrow \mathbb{R}$ to the nonlinear harmonic space $ (\Omega,\mathcal{H}_{\mathscr{L}_{p(.)}} )$ of the $ \mathscr{L}_{p(.)}$-harmonic functions in $ \Omega $. 

The following theorem suggests that the mapping $ f\mapsto {\mathscr{L}^{\Omega}_{p(.),f}}$ from $ \mathscr{C}(\partial\Omega) $ into $ \mathcal{H}_{\mathscr{L}_{p(.)}}(\Omega) $ is injective whenever $ \Omega $ is $p(.)$-regular in capacity.  
\begin{theorem}
	Let $\Omega\subset \mathbb{R}^{n}$ be a bounded open set with non empty boundary $ \partial\Omega$. Assume that $ \Omega $ is $p(.)$-regular in capacity. Then the mapping $ f\mapsto \mathscr{L}^{\Omega}_{p(.),f}$ from $ \mathscr{C}(\partial\Omega) $ into $ \mathcal{H}_{\mathscr{L}_{p(.)}}(\Omega) $ is injective.
\end{theorem}
\begin{corollary}
	Let $\Omega\subset \mathbb{R}^{n}$ be a bounded open set with non empty boundary $ \partial\Omega$. Let $\mathscr{R}(\partial\Omega)$ the class of all real valued resolutive function on 
	$\partial\Omega$ and $f,g\in  \mathscr{R}(\partial\Omega)$. Suppose that $ \Omega $ is $p(.)$-regular in capacity. If 
	$\mathscr{L}^{\Omega}_{p(.),f}=\mathscr{L}^{\Omega}_{p(.),g}$, then the set $ \left\lbrace x\in \partial \Omega: f(x)\neq g(x)\right\rbrace$ is negligible.   
\end{corollary}
We now give the structure of the paper. More detailed descriptions appear at the beginnings of the sections.

Section $ 2 $ gives sufficient conditions for the existence of a $p(.)$-relatively quasicontinuous representative for functions in $\tilde{\mathcal{W}}^{1,p(.)}(\Omega)$ and studies $p(.)$-fine continuity. 

Section $ 3 $ defines $p(.)$-regularity in capacity and $ p(.)$-zero property. This is used to give a news characterization of variable exponent Sobolev trace spaces modulo the technical results of Section $ 2 $. The Perron-Weiner-Brelot operator in nonlinear harmonic spaces is constructed in Section $ 4 $, which study the injectivity  of this operator.
\section{\textbf{ $p(.)$-fine topology of equivalence class of Sobolev functions}}
Fine topology  which dates back to Henri Cartan \cite{Cartan} plays a central role in the theory of Sobolev spaces and in potential theory. When finding a representative with certain continuity properties in an equivalence class of almost everywhere equal functions, the Euclidean topology is not relevant in general, instead one can use the fine topology. In addition, fine topology is used much more extensively in nonlinear potential theory  literature than the Euclidean topology, which is a reason why its theory has been more developed. For references to nonlinear and fine nonlinear potential theory, we refer to \cite{Adams,Bjorn 1,Evans,Fuglede,Hedberg,Heinonen 1,Heinonen,Kilpel 2}. 

\subsection{\textbf{ $p(.)$-relative quasicontinuous representative  of equivalence class of Sobolev functions}}
Sobolev functions are defined only up to Lebesgue measure zero and thus it is not always clear how to use their point-wise properties. But one can also think of some representative of this equivalence class, perhaps defined at all points outside a set of measure zero. However, if continuous functions are dense in the variable exponent Sobolev space, then each
function in $ \mathcal{W}^{1,p(.)}(\mathbb{R}^{n}) $ has a $p(.)$-quasicontinuous representative, see \cite[ Theorem 5.2]{Hasto2}. Even when the Hardy-Littlewood
maximal operator  $\mathscr{M}: L^{p(.)}(\mathbb{R}^{n})\longrightarrow
L^{p(.)}(\mathbb{R}^{n})$ is bounded, the \cite[ Theorem 5.2]{Harjulehto}
tels us that every equivalence class contains a H\"{o}lder continuous function. Moreover, there are trace theorems that give the existence of distinguished elements in the equivalence class, so that restrictions to some sets of zero Lebesgue measure, we refer to \cite{Hasto,Diening and Hasto}.

In this subsection, we show that the equivalent class of Sobolev functions  in $ \tilde{{\mathcal{W}}}^{1,p(.)}(\Omega)$ are relatively 
$ p(.)$-quasicontinuous. This is a concept with deep roots in nonlinear potential theory associated with variable exponent spaces. For further applications of the $p(.)$-quasicontinuity, we refer the reader to \cite{Harjulehto 1} and the references therein.
\subsubsection{\textbf{ $p(.)$-Capacities}}
We begin by recalling the definition of the $p(.)$-Sobolev capacity appearing in the existing literature. We refer to \cite{Hasto,Hasto2,Hasto3}.\\
Let $p\in \mathcal{P}(\Omega)$. For a set $ E\subset \mathbb{R}^{n}$, we denote
\begin{eqnarray*}	
	\mathcal{A}_{p(.)}(E):=\left\lbrace u\in \mathcal{W}^{1,p(.)}(\mathbb{R}^{n}): u\geq 1 \ \textnormal{a.e. on a neighbourhood of } E \right\rbrace.    	
\end{eqnarray*}
The Sobolev $p(.)$-capacity of the set $E$ is the number defined by
\begin{eqnarray*}
	C_{p(.)}(E):=\inf_{u\in\mathcal{A}_{p(.)}(E)}\rho_{1,p(.)}(u).	   	
\end{eqnarray*}
In case $\mathcal{A}_{p(.)}(E)=\varnothing$, we set $C_{p(.)}(E)=\infty$.  Under assumption $ 1 < p^{-} \leq p^{+} < \infty $ the Sobolev $p(.)$-capacity  is an outer measure and a Choquet capacity. 

Next we present another version of the relative $ p(.)$-capacity. For $ p\in \mathcal{P}(\Omega)$, we let
\begin{eqnarray*}
	\tilde{{\mathcal{W}}}^{1,p(.)}(\Omega):= \overline{{\mathcal{W}}^{1,p(.)}(\Omega)\cap \mathscr{C}_{c}(\overline{\Omega})}^{{\mathcal{W}}^{1,p(.)}(\Omega)}
\end{eqnarray*}
and
\begin{eqnarray*}
	\mathcal{R}^{ \overline{\Omega}}_{p(.)}(O):=\left\lbrace u\in \tilde{{\mathcal{W}}}^{1,p(.)}(\Omega): u\geq 1 \ \textnormal{ a.e. on}\ O \right\rbrace.    	
\end{eqnarray*}

\begin{definition}
Let $O\subset \overline{\Omega}$ be a relatively open set, that is, open with respect to the relative topology of $\overline{\Omega}$. 
We define the relative $p(.)$-capacity of $O$, with respect to $\Omega$, by
\begin{eqnarray*}
	C_{p(.)}^{\overline{\Omega}}(O):=\inf_{u\in\mathcal{R}^{ \overline{\Omega}}_{p(.)}(O)}{\rho}_{1,p(.)}^{\Omega}(u).	   	
\end{eqnarray*}
For any set $ E\subset \overline{\Omega}$,
\begin{eqnarray*} 
	C_{p(.)}^{\overline{\Omega}}(E)=\inf\left\lbrace C_{p(.)}^{\overline{\Omega}}(O): O \ \textnormal{relatively open in}\ \overline{\Omega} \ \textnormal{containing}\ E  \right\rbrace. 
\end{eqnarray*}
\end{definition}
Notice that our $p(.)$-relative capacity  $ C_{p(.)}^{\overline{\Omega}}$ defined with the space $\tilde{{\mathcal{W}}}^{1,p(.)}(\Omega)$ slightly differs from the one introduced in \cite{Hasto,Hasto4}. However, the resulting capacities are equivalent, and hence for our purposes the difference is irrelevant.

In this paper, the motivation to use the $p(.)$-relative capacity  $ C_{p(.)}^{\overline{\Omega}}$ is to study
fine properties of equivalence class of Sobolev functions in $\tilde{{\mathcal{W}}}^{1,p(.)}(\Omega)$ and to give a news characterization of variable exponent Sobolev trace spaces with application to  Dirichlet problem in nonlinear harmonic space.

Recall from \cite{BA} that the set function $E\mapsto C_{p(.)}^{\overline{\Omega}}$ has the following properties:
 \begin{itemize}
 	\item $(C_{1})$ $C_{p(.)}^{\overline{\Omega}}(\varnothing)=0$;
 	\item $(C_{2})$ If $ E_{1}\subset E_{2} \subset \Omega_{2} \subset \Omega_{1}$, then
 	\begin{equation*}
 	C_{p(.)}^{\overline{\Omega}_{1}}(E_{1})\leqslant 	C_{p(.)}^{\overline{\Omega}_{1}}(E_{2});
 	\end{equation*}
 	\item $(C_{3})$ If $K_{1}\supset K_{2}\supset K_{3}\dots$ are compact subsets of $\Omega$, then
 	\begin{equation*}
 C_{p(.)}^{\overline{\Omega}}(\cap_{i=1}^{\infty} K_{i})=\lim _{i\rightarrow\infty}C_{p(.)}^{\overline{\Omega}}(K_{i}); 
 	\end{equation*}
 	\item $(C_{4})$ If $E_{1}\subset E_{2}\dots$ are subsets of $\Omega$, then
 	\begin{equation*}
 	C_{p(.)}^{\overline{\Omega}}(\cup_{i=1}^{\infty} E_{i})=\lim _{i\rightarrow\infty}C_{p(.)}^{\overline{\Omega}}(E_{i}); 
 	\end{equation*}
 	\item $(C_{5})$ For $E_{i}\subset \Omega$, $ i\in \mathbb{N}$, we have
 	\begin{equation*}
 	C_{p(.)}^{\overline{\Omega}}(\cup_{i=1}^{\infty} E_{i})\leqslant \sum ^{\infty}_{i=1} C_{p(.)}^{\overline{\Omega}}(E_{i}).
 	\end{equation*}
 \end{itemize}
 This means that the $p(.)$-relative capacity $ C_{p(.)}^{\overline{\Omega}}$ is an outer measure and a Choquet capacity.  
 \begin{proposition}
 	Let $F$ be a closed subset of $\Omega$. Then 	
 	\begin{equation*}
 	C_{p(.)}^{\overline{\Omega}}(\partial F)=C_{p(.)}^{\overline{\Omega}}(F)
 	\end{equation*}	
 \end{proposition}
 \begin{proof}
 Notice that since $ \mathcal{W} ^{1,p(.)}_{0}(\Omega)\subset  \tilde{{\mathcal{W}}}^{1,p(.)}(\Omega)$, then for every set $ E\subset\Omega $,	
 	\begin{equation*}
 	C_{p(.)}^{\overline{\Omega}}(E)=\inf \int_{\Omega}\left|\nabla u(x)\right|^{p(x)} \ dx, 
 	\end{equation*}	
 	where the infimum is taken over all $ u\in \mathcal{W}^{1,p(.)}_{0}(\Omega) $ (extended by zero outside $ \Omega $) which are at least one in a neighborhood of $ E $.\\
 	Let $ u $ be a $p(.)$-admissible function in the definition of 
 	$C_{p(.)}^{\overline{\Omega}}(\partial F)$, with $ 0\leqslant u \leqslant 1$ and $ u=0 $ on $\mathbb{R}^{n}\setminus \Omega$. Let 
 	\begin{equation*}	
 	v:=
 	\left\{
 	\begin{array}{rl}
 	1 & \text{in} \ F ,\\
 	u& \text{in} \ \mathbb{R}^{n}\setminus F. 
 	\end{array}
 	\right.
 	\end{equation*}
 	Then
 	\begin{eqnarray*}
 		\left\| v \right\|_{W^{1,p(.)}(\Omega)}&=&\left\| v \right\|_{L^{p(.)}(\Omega)}+\left\| \nabla v \right\|_{L^{p(.)}(\Omega)}\\
 		&\leqslant& \left\| u \right\|_{W^{1,p(.)}(\Omega)}+2\max \left\lbrace \left|F \right|^{\frac{1}{p^{+}}}, \left|F \right|^{\frac{1}{p^{-}}}  \right\rbrace\\
 		&<\infty&.    
 	\end{eqnarray*}
 	Hence $ v\in \mathcal{W}^{1,p(.)}(\Omega) $. Since $ v=u=0 $ in $ \mathbb{R}^{n}\setminus \Omega $, we have that $ v\in \mathcal{W}^{1,p(.)}_{0}(\Omega) $ and 
 	\begin{eqnarray*}
 		C_{p(.)}^{\overline{\Omega}}(F)&=&\inf \int_{\Omega}\left|\nabla (v(x))\right|^{p(x)} \ dx\\
 		&\leqslant& \inf \int_{\Omega}\left|\nabla (u(x))\right|^{p(x)} \ dx. 		
 	\end{eqnarray*}
 	Taking infimum over all $ u $, we get that
 	\begin{equation*}
 	C_{p(.)}^{\overline{\Omega}}(F)\leqslant C_{p(.)}^{\overline{\Omega}}(\partial F).
 	\end{equation*}
 	To prove the converse inequality, let $ F $ be a closed subset of $ \Omega $. Then $ \partial F \subset F $ and 
 	\begin{equation*}
 	C_{p(.)}^{\overline{\Omega}}(\partial F)\leqslant C_{p(.)}^{\overline{\Omega}}(F).
 	\end{equation*}  	
 \end{proof}
\subsubsection{\textbf{ $p(.)$-relatively quasicontinuous representative}}
In this sub subsection we study  $p(.)$-relative quasicontinuity property of Sobolev functions. It turns out that Sobolev functions are defined up to a set of $p(.)$-relative capacity zero.
\begin{definition}
	\item A set $P\subset\Omega$ is called $p(.)$-relatively polar if $C_{p(.)}^{\overline{\Omega}}(P)=0$.
\item We say that a property holds on a  set $ A\subset\Omega$ $ p(.) $-relatively quasieverywhere ($ p(.) $-r.q.e., for short) if there exists a $p(.)$-relatively polar set $P\subset A$ such that the property holds everywhere on $ A\setminus P$.
\end{definition}
By definition $\mathcal{W}^{1,p(.)}(\Omega)\cap \mathscr{C}_{c}(\overline{\Omega})$ is dense in $ \tilde{{\mathcal{W}}}^{1,p(.)}(\Omega)$ which is complete Banach space. The next result gives a way to find a $ p(.) $-relatively quasieverywhere converging subsequence.
\begin{theorem}\label{subsequence which converges $p(.)$-r.q.e.}
	Let $p\in\mathcal{P}(\Omega)$ satisfy $ 1<p^{-}\leqslant p^{+}<\infty$. For each Cauchy sequence with respect to the $\mathcal{W}^{1,p(.)}(\Omega)$-norm of functions in $\tilde{\mathcal{W}}^{1,p(.)}(\Omega)$ there exists a subsequence which converges $p(.)$-r.q.e. in $\Omega$. Moreover, the convergence is uniform outside a set of arbitrary small relative $p(.)$-capacity.
\end{theorem}
\begin{proof}
	Let $(u_{i})$ be a Cauchy sequence in $\tilde{\mathcal{W}}^{1,p(.)}(\Omega)$. Without loss of generality, we denote again by $(u_{i})$ the subsequence of $(u_{i})$ such that
	\begin{eqnarray*}
		\left\| u_{i+1}-u_{i} \right\|_{\mathcal{W}^{1,p(.)}(\Omega)}\leq \dfrac{1}{8^{i}}, \ i\in \mathbb{N}.   
	\end{eqnarray*}
	Put
	\begin{eqnarray*}
		G_{i}:=\left\lbrace x\in\Omega: \ \left|u_{i+1}(x)-u_{i}(x) \right|>\dfrac{1}{2^{i}}, \ i\in \mathbb{N}\right\rbrace   
	\end{eqnarray*}
	and 
	\begin{eqnarray*}
		G_{k}:=\bigcup_{i=k}^{\infty}G_{i}.  
	\end{eqnarray*}
	Hence $G_{i}$ is an open set in $\Omega$ and $ 2^{i} \left|u_{i+1}(x)-u_{i}(x) \right|>1$ on $G_{i}$. Since 
	$\mathcal{W}^{1,p(.)}(\Omega)$ is a Banach lattice, we deduce that $\tilde{\mathcal{W}}^{1,p(.)}(\Omega)$ is also a Banach lattice. Therefore, 
	\begin{eqnarray*}
		2^{i} \left|u_{i+1}(x)-u_{i}(x) \right|\in\tilde{\mathcal{W}}^{1,p(.)}(\Omega)  
		\ \textnormal{and}\ \left\| 2^{i} \left|u_{i+1}(x)-u_{i}(x) \right|\right\|_{\mathcal{W}^{1,p(.)}(\Omega)}\leqslant \dfrac{1}{4^{i}}\leqslant 1. 
	\end{eqnarray*}
	By the unit ball property (\cite[Lemma 2.1.14]{Hasto}) we get that
	\begin{eqnarray*}
		\rho_{1,p(.)}^{\Omega}(2^{i} \left|u_{i+1}(x)-u_{i}(x) \right|)\leqslant 
		2^{i}\left\| u_{i+1}-u_{i}\right\|_{W^{1,p(.)}(\Omega)}\leqslant \dfrac{1}{4^{i}}. 
	\end{eqnarray*}
	Consequently
	\begin{eqnarray*}
		C_{p(.)}^{\overline{\Omega}}(G_{i})\leqslant \dfrac{1}{4^{i}}.
	\end{eqnarray*}
	By sub-additivity of the relative $p(.)$-capacity, we obtain that
	\begin{eqnarray*}
		C_{p(.)}^{\overline{\Omega}}(G_{k})&=&C_{p(.)}^{\overline{\Omega}}(\cup^{\infty}_{i=k} G_{i})\\
		&\leqslant& \sum_{i=k}^{\infty}C_{p(.)}^{\overline{\Omega}}(G_{i})\\
		&\leqslant& \sum_{i=k}^{\infty} \dfrac{1}{4^{i}}=\dfrac{1}{4^{k-1}}.
	\end{eqnarray*}
	Hence
	\begin{eqnarray*}
		C_{p(.)}^{\overline{\Omega}}(\cap_{k=1}^{\infty} G_{k})&=&C_{p(.)}^{\overline{\Omega}}(\cap_{k=1}^{\infty} \cup^{\infty}_{i=k} G_{i})\\
		&\leqslant& \lim_{k\rightarrow \infty}C_{p(.)}^{\overline{\Omega}}(\cup_{i=k}^{\infty}G_{i})\\
		&\leqslant& \lim_{k\rightarrow \infty} \dfrac{1}{4^{k-1}}=0.
	\end{eqnarray*} 
	Thus 
	\begin{eqnarray*}
		C_{p(.)}^{\overline{\Omega}}(\cap_{k=1}^{\infty} G_{k})=0.
	\end{eqnarray*}
	Consequently $\cap_{k=1}^{\infty} \cup^{\infty}_{i=k} G_{i}  $ is a $p(.)$-relatively polar set. Moreover $u_{i}$ converges pointwise in $ \Omega\setminus \cap_{k=1}^{\infty} \cup^{\infty}_{i=k} G_{i}  $. Since $ \left|u_{i+1}(x)-u_{i}(x) \right|\leqslant \dfrac{1}{2^{i}}   $ in $ \Omega\setminus \cap_{k=1}^{\infty} \cup^{\infty}_{i=k} G_{i}  $ for all $i\geq k $, we have that $(u_{i})$ is a sequence of continuous functions on $\Omega$ which converges uniformly in $\Omega\setminus G_{k}$.          
\end{proof}
In the following we give sufficient conditions for the existence of a $p(.)$-relatively quasicontinuous representative for functions in $\tilde{\mathcal{W}}^{1,p(.)}(\Omega)$.
\begin{definition}
	\item  A function $ u:\Omega\longrightarrow \mathbb{R}$ is said to be $p(.)$-relatively quasicontinuous ($ p(.) $-r.q.c., for short) if for every $ \varepsilon>0 $, there exists a relatively open set $ O_{\varepsilon}\subset\Omega $ such that $C_{p(.)}^{\overline{\Omega}}(O_{\varepsilon})<\varepsilon$ and $ u $ is continuous on $ \Omega\setminus O_{\varepsilon}$.    
\end{definition}
\begin{theorem}\label{quasicontinuous representative}
	Let $p\in\mathcal{P}(\Omega)$ with $ 1<p^{-}\leqslant p^{+}<\infty$. Then for every $ u\in\tilde{\mathcal{W}}^{1,p(.)}(\Omega) $, there exists a unique (up to a $p(.)$-relative polar set) $p(.)$-r.q.c. function $\tilde{u}:\Omega\longrightarrow \mathbb{R}$ such that $\tilde{u}=u$ a.e. in $\Omega$. 
\end{theorem}
\begin{proof}
	Let $u\in\tilde{\mathcal{W}}^{1,p(.)}(\Omega)$. There exists a sequence $\mathcal{W}^{1,p(.)}(\Omega)\cap \mathscr{C}_{c}(\overline{\Omega})$ such that 
	$u_{i}\longrightarrow u$ in $\mathcal{W}^{1,p(.)}(\Omega)$. By Theorem \ref{subsequence which converges $p(.)$-r.q.e.} there exists a subsequence which converges $p(.)$-r.q.e. in $ \Omega $ and uniformly outside a $p(.)$-relative polar set. Let $\tilde{u}\longrightarrow u_{i}$ be the point-wise limit of 
	$(u_{i})$. By the uniform convergence we get that $\tilde{u}:\Omega\longrightarrow \mathbb{R}$ is $p(.)$-r.q.c. $\tilde{u}=u$ a.e. in $\Omega$. For the uniqueness, we assume that there exists another $\tilde{v}$ $p(.)$-r.q.c. in $\Omega$ such that $\tilde{v}=u$ a.e. in $\Omega$. Hence $\tilde{u}-\tilde{v}=0$ a.e. in $\Omega$ and $\tilde{u}-\tilde{v}=0$ is $p(.)$-r.q.c. in $\Omega$.    
\end{proof}
\begin{corollary}
	Let $p\in\mathcal{P}(\Omega)$ satisfy $ 1<p^{-}\leqslant p^{+}<\infty$. Let $(u_{i})$ be a sequence of $p(.)$-r.q.c. functions in $\tilde{\mathcal{W}}^{1,p(.)}(\Omega)$ which converges to a $p(.)$-r.q.c. function $u\in\tilde{\mathcal{W}}^{1,p(.)}(\Omega)$. Then there exists a subsequence which converges
	$p(.)$-r.q.e. to $u$ on $\Omega$.
\end{corollary}
\begin{proof}
	Let $ (u_{i_{k}} )$ be a subsequence of $ (u_{i} )$ such that
	\begin{eqnarray*}
		\sum_{k=1}^{\infty} 2^{i_{k}}\left\| u_{i_{k}}-u \right\|_{\mathcal{W}^{1,p(.)}(\Omega)}\leq 1.
	\end{eqnarray*} 
	Put
	\begin{eqnarray*}
		P:=\cap_{j=1}^{\infty} \cup^{\infty}_{k=j} G_{k}, 
	\end{eqnarray*}
	where 
	\begin{eqnarray*}
		G_{k}:=\left\lbrace x\in\Omega: \ \left|u_{i_{k}}(x)-u(x) \right|>\dfrac{1}{2^{i_{k}}}.\right\rbrace. 
	\end{eqnarray*}
	There exists $ j_{0}\in \mathbb{N}$ such that  
	\begin{eqnarray*}
		\left|u_{i_{k}}(x)-u(x) \right|\leqslant \dfrac{1}{2^{i_{k}}},\ \forall \ k\geq j_{0}.  
	\end{eqnarray*}
	Hence $u_{i_{k}}(x)$ converges uniformly in $ \Omega\setminus \cup _{k=j_{0}}^{\infty}G_{k} $ and everywhere in $ \Omega\setminus P $. By the same way in the proof of Theorem \ref{subsequence which converges $p(.)$-r.q.e.} we get that
	\begin{eqnarray*}
		C_{p(.)}^{\overline{\Omega}}(P)=0.
	\end{eqnarray*}
	Hence $P$ is $p(.)$-relatively polar set and the proof is finished.
\end{proof}

\subsubsection{\textbf{ $p(.)$-fine continuity}}
In addition to the $ p(.)$-relative quasicontinuity studied in this section, the equivalent class of Sobolev functions enjoy another, subtler continuity property, called $p(.)$-fine continuity.

$p(.)$-fine continuity is closely related to the concept of a $p(.)$-thin set, which is the subject of the following definition.
\begin{definition}
Let $E\subset \mathbb{R}^{n}$ and $ p\in \mathcal{P}(\Omega)$. The set $ E $ is called $p(.)$-thin at a point $a\in \mathbb{R}^{n}$ if
\begin{equation*}
\int_{0}^{1} \left( \dfrac{C_{p(.)}^{\overline{B(a,2r)}}(E\cap B(a,r))}{C_{p(.)}^{\overline{B(a,2r)}}(B(a,r))}\right)^{p^{'}(a)-1}  \ \dfrac{dr}{r}<\infty.
\end{equation*}
We say that $ E $ is $p(.)$-thick at $ a $ if $ E $ is not $ p(.) $-thin
at $ a $.
\end{definition} 
We follow the outlines given in \cite[ Theorem 6.33] {Heinonen}, we deduce that if $a\in E$ and $ C_{p(.)}^{\overline{\Omega}}\left\lbrace a\right\rbrace>0 $, then $ E $ is $p(.)$-thick at $ a $. On the other hand, if $a\in \overline{E}\setminus E$, then $ E $ can be $p(.)$-thin at $a$ even if $ C_{p(.)}^{\overline{\Omega}}\left\lbrace a\right\rbrace>0 $. Therefore, if $E$ consists of a sequence $ x_{i}\neq 0$ converging to $ 0 $, then $ E $ is $p(.)$-thin at $a$.

It is useful to define an $p(.)$-fine topology associated to the concept of $p(.)$-thinness.
\begin{definition}
 Let $x \in \mathbb{R}^{n}$. Then a set $ U\subset \mathbb{R}^{n}$ is called an $p(.)$-fine neighborhood of $ x $ if $ x \in U $ and $U^{c}$ is $p(.)$-thin at $ x $.
 
A set $ G\subset \mathbb{R}^{n}$ is $p(.)$-finely open if it is an $p(.)$-fine neighborhood of each of its points.

A set $ F\subset \mathbb{R}^{n}$ is $p(.)$-finely closed if $ F^{c} $ is $p(.)$-finely open.
\end{definition}
This definition give a rise to a topology wich we call $p(.)$-fine topology. Its clear that $p(.)$-fine topology is finer than the Euclidean topology.
\begin{definition}\label{definition fine continuity}
 A real function $ u $ defined in a $p(.)$-fine open set $ U $ is $p(.)$-finely continuous at a point $ a\in U $ if the set
 $ \left\lbrace x\in U: \left|f(x)-f(a) \right|\geq \varepsilon \right\rbrace $ is $p(.)$-thin at $ a $ for each $ \varepsilon >0 $.    	
\end{definition}
 It follows from  Definition \ref{definition fine continuity} that, if $ u$ is defined on $ F\subset G $ such that $C_{p(.)}^{\overline{\Omega}}(G\setminus F)=0$, then $ u $ is $p(.)$-finely continuous at $ x\in F $ if and only if the set
 \begin{equation*}
  \left\lbrace y\in F: \left|u(y)-u(x) \right|\geq \varepsilon \right\rbrace \cup F^{c}
 \end{equation*}  
   is $p(.)$-thin at $ x $ for each $\varepsilon >0$.  
 \begin{theorem}\label{p(.)-fine continuity}
Let $u\in\tilde{\mathcal{W}}^{1,p(.)}(\Omega)$. Then $ u $ is $p(.)$-finely continuous relatively quasieverywhere in $ \Omega $. 
 \end{theorem}
\begin{proof}
Let $u\in\tilde{\mathcal{W}}^{1,p(.)}(\Omega)$. By Theorem \ref{quasicontinuous representative}, there exist a $p(.)$-r.q.c representative $\tilde{u}$ such that $\tilde{u}=u$ a.e. in $\Omega$. Hence for every $ \varepsilon >0 $, there is a decreasing sequence of open sets, $ (G_{n} )_{n\geq 1}\subset\Omega$ such that $ C_{p(.)}^{\overline{\Omega}}(G_{n})<\varepsilon$ and $ u $ is continuous on $ \Omega\setminus G_{n}$. In particular $\lim_{n\longrightarrow +\infty} G_{n}=0$. For each $ n=1,2,3,\ldots$, let $ u_{n} $ be a $p(.)$-admissible functions in the definition of $ C_{p(.)}^{\overline{\Omega}}(G_{n}) $. Hence $ u_{n}\geq 1$ a.e. in $ G_{n} $ and by Theorem \ref{subsequence which converges $p(.)$-r.q.e.}, we can assume that $\lim_{n\longrightarrow +\infty} u_{n}=0 \ p(.)\text{-r.q.e.}$, therefore, for all sufficiently large $ n $, we give that $G_{n}$ is $p(.)$-thin. Hence $ u $ is $p(.)$-finely continuous relatively quasieverywhere in $ \Omega $.    
\end{proof}
\begin{corollary}
Let $u\in\tilde{\mathcal{W}}^{1,p(.)}(\Omega)$. Then there is a set 
$E\subset\Omega$ such that $ E $ is $p(.)$-thin at $ x\in\Omega $ and \begin{equation*}
\lim_{y\longrightarrow x} u(y)=u(x),	
\end{equation*}
for all $ y\in \Omega\setminus E $.
\end{corollary}
\begin{proof}
Let $u\in\tilde{\mathcal{W}}^{1,p(.)}(\Omega)$. By Theorem \ref{p(.)-fine continuity}, $ u $ is $p(.)$-finely continuous relatively quasieverywhere in $ \Omega $. Then for each $ n $, $ n=1,2,\ldots$, the set 
\begin{equation*}
	E_{n}=\left\lbrace y: y\in\Omega,\left|u(y)-u(x) \right|\geq \dfrac{1}{n}  \right\rbrace 
\end{equation*}
is $p(.)$-thin at $x\in\Omega$. Moreover, there are a positive numbers $ r_{n} $ so small such that 
\begin{equation*}
\int_{0}^{1} \left( \dfrac{C_{p(.)}^{\overline{B(x,2r)}}(E_{n}\cap B(x,r_{n}) \cap B(x,r))}{C_{p(.)}^{\overline{B(x,2r)}}(B(x,r))}\right)^{p^{'}(x)-1}  \ \dfrac{dr}{r}<\dfrac{1}{2^{n}}.
\end{equation*}
Then the set 
\begin{equation*}
E:=\bigcup_{n=1}^{\infty}(E_{n}\cap B(x,r_{n}))
\end{equation*}
is $p(.)$-thin at $x\in\Omega$, and for all $ y\in \Omega\setminus E $, we have that 
 \begin{equation*}
 \left|u(y)-u(x) \right| < \dfrac{1}{n}.
 \end{equation*}
 Then the set $ E $ is the desired subset of $ \Omega $.
\end{proof}       	   
\section{{\textbf{ Variable exposent Sobolev trace spaces }}}
We now turn our attention in the question of traces of Sobolev functions on the boundary of the set of definition. This problem is more delicate than the interior one, since under some regularity assumptions on the variable exponent $ p$ it is always possible to approximate a Sobolev function in the space $ \mathcal{W}^{1,p(.)}$ by smooth function; see \cite{Diening, Edmunds 2,FAN 1,Hasto5,Samko,Zhikov,Zhikov1}, the same is not true up to the boundary. Let $\Omega\subset \mathbb{R}^{n+1}$ be a Lipschitz domain and let $F\in \mathcal{W}^{1,p(.)}(\Omega)$ be a function. In a neighborhood of a boundary point $ x\in \partial\Omega$ we have a local bilipschitz chart which maps part of the boundary to $\mathbb{R}^{n}$ and transported back to $\partial\Omega$ using the inverse chart. Thus according to what was explained in \cite{Diening and Hasto,Diening and Hasto 1,Hasto}, the trace $Tr F$ of $F$ is defined as a function in $L_{loc}^{1}(\partial\Omega)$. Note that if $\mathcal{W}^{1,p(.)}(\Omega)\cap \mathscr{C}(\overline{\Omega})$, then 
$Tr F=F_{|\partial\Omega}$. The trace space 
$Tr \mathcal{W}^{1,p(.)}(\Omega)$ consists of the traces of all functions $F\in \mathcal{W}^{1,p(.)}(\Omega)$. The elements of 
$Tr \mathcal{W}^{1,p(.)}(\Omega)$ are functions on $\partial\Omega$. The quotient norm
\begin{eqnarray*}
	\left\|f \right\|_{Tr \mathcal{W}^{1,p(.)}(\Omega)}:=\inf \left\lbrace \left\|F \right\|_{\mathcal{W}^{1,p(.)}(\Omega)}:\ F\in \mathcal{W}^{1,p(.)}(\Omega)\ \textnormal{and} \ Tr F=f  \right\rbrace  	   
\end{eqnarray*}
makes Tr $\mathcal{W}^{1,p(.)}(\Omega)$ a Banach space. Moreover, we have the following characterization of $\mathcal{W}^{1,p(.)}_{0}(\Omega)$ in terms of traces: 
\begin{center}
	Let $F\in \mathcal{W}^{1,p(.)}(\Omega)$, then $F\in \mathcal{W}^{1,p(.)}_{0}(\Omega)$ if and only if $TrF=0 $.
\end{center}
In particular,  
\begin{eqnarray*}
	F\in\mathcal{W}^{1,p(.)}(\Omega)\cap \mathscr{C}(\overline{\Omega}) \ \textnormal{implies that} \ F(x)=0 \ \textnormal{for every} \ x\in\partial\Omega.
\end{eqnarray*} 
This property is not directly related to the fact that $ u\in \mathcal{W}^{1,p(.)}(\Omega)$, but to the geometry of $ \Omega$.  Notice that the Besov spaces can also be used to describe the traces of Sobolev functions, see \cite{Hasto}. For further results on variable exponent Sobolev trace spaces, see  \cite{Edmunds, Edmunds 1}.

 In view of potential theory a Sobolev function has a distinguished representative which is defined up to a set of capacity zero \cite{Hasto, Hasto2, Hasto3, Hasto4}. Therefore it is possible to look at traces of Sobolev functions on the boundary of the set of definition.
 According to Theorem \ref{quasicontinuous representative} every function $ u $ in $\tilde{\mathcal{W}}^{1,p(.)}(\Omega)$ is defined $p(.)$-relatively quasieverywhere. Moreover, if $ E \subset \mathbb{R}^{n} $ with $C_{p(.)}^{\overline{\Omega}}(E)>0$ then the trace of $ u $ to $ E $ is the restriction to $ E $ of any $p(.)$-relatively quasicontinuous representative of $ u $.     

The purpose of this section is to give a news characterization of variable exponent Sobolev zero trace space $ \mathcal{W}^{1,p(.)}_{0}(\Omega)$. Let us beginning by the following lemma.
\begin{lemma}\label{funct with comparct supp is in W1,P}
	Let $u\in \mathcal{W}^{1,p(.)}(\Omega)$. Assume that $u$ has compact support in $\Omega$. Then $ u\in \mathcal{W}^{1,p(.)}_{0}(\Omega)$.
\end{lemma}  
\begin{proof}
	Let $u\in \mathcal{W}^{1,p(.)}(\Omega)$ and let $\psi\in \mathscr{C}^{\infty}_{0}(\Omega)$ be such that $\psi=1$ on the support of $u$. If a sequence $\psi_{j}$ converges to $u$ in $u\in \mathcal{W}^{1,p(.)}(\Omega)$, then $\psi \psi_{j}$ converges to 
	$\psi u=u$ in $\mathcal{W}^{1,p(.)}(\Omega)$. Thus $ \psi u\in \mathcal{W}^{1,p(.)}_{0}(\Omega)$ and hence $u\in \mathcal{W}^{1,p(.)}_{0}(\Omega)$.    	
\end{proof}
Let us looking $ \mathcal{W}^{1,p(.)}_{0}(\Omega)$ as an ideal of $ \mathcal{W}^{1,p(.)}(\Omega)$. A subspace $I$ of $\mathcal{W}^{1,p(.)}(\Omega)$ is called an ideal if for 
$u\in I,\ v\in \mathcal{W}^{1,p(.)}(\Omega)$, $ \left|v \right| \leqslant \left|u \right|$ a.e. implies that $v\in I$. The closed lattice ideals of the Sobolev spaces $\mathcal{W}^{1,p(.)}(\Omega)$ are those subspaces which consist of all functions which vanish on a prescribed set. To be precise, we have the following result:
\begin{theorem}\label{closed ideal}
	Let $ 1<p^{-}\leqslant p^{+}<\infty$. Then the space $\mathcal{W}^{1,p(.)}_{0}(\Omega)$ is a closed ideal in $\mathcal{W}^{1,p(.)}(\Omega)$. Moreover, there exists a borel set 
	$\mathcal{B}$ such that $$\mathcal{W}^{1,p(.)}_{0}(\Omega)=\mathcal{W}^{1,p(.)}_{0}(\mathcal{B}).$$
\end{theorem}
\begin{proof}
	$\mathcal{W}^{1,p(.)}_{0}(\Omega)$ is a closed Banach subspace of $\mathcal{W}^{1,p(.)}(\Omega)$ . Let $u\in \mathcal{W}^{1,p(.)}_{0}(\Omega),\ v\in \mathcal{W}^{1,p(.)}(\Omega)$, $0 \leqslant\left|v \right| \leqslant \left|u \right|$ a.e. Let 
	$\varphi_{n}\in \mathscr{C}^{\infty}_{0}(\Omega)$ such that $\varphi_{n}$ converges to $u$ in $\mathcal{W}^{1,p(.)}_{0}(\Omega)$. Then $ v_{n}:=v\wedge \varphi_{n}$ has compact support and belongs to $\mathcal{W}^{1,p(.)}(\Omega)$. Hence by Lemma \ref{funct with comparct supp is in W1,P}, we have that 
	$v_{n}\in \mathcal{W}^{1,p(.)}_{0}(\Omega)$. Moreover, since $ v_{n}\longrightarrow v\wedge u=v$ in $\mathcal{W}^{1,p(.)}(\Omega)$, we have that $v\in \mathcal{W}^{1,p(.)}_{0}(\Omega)$.\\
	To prove that there exists a borel set 
	$\mathcal{B}$ such that $\mathcal{W}^{1,p(.)}_{0}(\Omega)=\mathcal{W}^{1,p(.)}_{0}(\mathcal{B})$, it suffices to observe two basic facts:
	\begin{itemize}
		\item $\mathcal{W}^{1,p(.)}(\Omega)$ is a Banach lattice, hence 
		$f^{+}\in \mathcal{W}^{1,p(.)}(\Omega)$ and $f \wedge 1\in \mathcal{W}^{1,p(.)}(\Omega)$, whenever $f\in \mathcal{W}^{1,p(.)}(\Omega)$.
		\item By using the H\"{o}lder inequality's, we get that 
		$f,g \in \mathcal{W}^{1,p(.)}(\Omega)\cap L^{\infty}(\Omega)$ implies that
		\begin{eqnarray*}
			\left\|fg\right\|_{\mathcal{W}^{1,p(.)}(\Omega)}\leq 2 	\left\|f\right\|_{\infty}	\left\|g\right\|_{\mathcal{W}^{1,p(.)}(\Omega)}+2\left\|g\right\|_{\infty}	\left\|f\right\|_{\mathcal{W}^{1,p(.)}(\Omega)}.   		
		\end{eqnarray*}    
	\end{itemize}
	Hence the same proof as \cite{Stollmann} yields.
\end{proof}  

Now we can describe $\mathcal{W}^{1,p(.)} _{0}(\Omega)$ as a subspace of $\tilde{\mathcal{W}}^{1,p(.)}(\Omega)$ in the following way:  
\begin{theorem}\label{description of W^{1,p}_{0} by relative capacity}
Let $\Omega\subset \mathbb{R}^{n}$ be an open set. Then
\begin{eqnarray*}
\mathcal{W}^{1,p(.)}_{0}(\Omega)=\left\lbrace u\in \tilde{\mathcal{W}}^{1,p(.)}(\Omega):\ \tilde{u}=0 \ p(.)\textnormal{-r.q.e. on} \ \partial\Omega   \right\rbrace, 	
\end{eqnarray*} 
where $\tilde{u}$ denotes the relatively $p(.)$-quasicontinuous representative of $u$.
\end{theorem}
\begin{proof}
From Theorem \ref{closed ideal}, we deduce that $\mathcal{W}^{1,p(.)} _{0}(\Omega)$ is also a closed ideal of $\tilde{\mathcal{W}}^{1,p(.)}(\Omega)$. Hence there exist a Borel set 
$\mathcal{B}\subset \Omega$ such that
\begin{eqnarray*}
	\mathcal{W}^{1,p(.)}_{0}(\Omega)=\left\lbrace u\in \tilde{\mathcal{W}}^{1,p(.)}(\Omega):\ \tilde{u}=0 \ p(.)\textnormal{-r.q.e on} \ \mathcal{B} \right\rbrace. 	
\end{eqnarray*}
\begin{itemize}
	\item \textbf{Step 1}: we show that $\mathcal{B}\cap\Omega$ is $p(.)$-relatively polar set.\\
	Let $(K_{i})_{i}$ be a sequence of compact subset of 
	$\Omega$ such that $K_{i}\subset K_{i+1}$ and $ \Omega=\bigcup_{i\in\mathbb{N}}K_{i}$. For each $i\in\mathbb{N}  $ there exist a cut-off function $ \varphi_{i}\in \mathscr{C}^{\infty}_{0}(\Omega)$ such that 
	$0\leqslant\varphi_{i}(x)\leqslant 1$ and $\varphi_{i}(x)=1$ on $K_{i}$. Since $\mathscr{C}^{\infty}_{0}(\Omega)\subset \mathcal{W}^{1,p(.)} _{0}(\Omega)$ it follows that $ \varphi_{i}\in \mathcal{W}^{1,p(.)} _{0}(\Omega)$ and $\varphi_{i}(x)=0\ p(.)$-r.q.e on $\mathcal{B}$. Hence $ C_{p(.)}^{\overline{\Omega}}(K_{i}\cap \mathcal{B})=0$. Consequently
\begin{eqnarray*}
C_{p(.)}^{\overline{\Omega}}(\mathcal{B}\cap \Omega)&=&C_{p(.)}^{\overline{\Omega}}(\mathcal{B}\cap \bigcup_{i\in\mathbb{N}}K_{i})\\
&=&\lim_{i\longrightarrow\infty} C_{p(.)}^{\overline{\Omega}}(\mathcal{B}\cap K_{i})\\
&=&0.	
\end{eqnarray*} 
Hence $\mathcal{B}\cap\Omega$ is $p(.)$-relatively polar set.
	\item \textbf{Step 2}: let now $ u\in \tilde{\mathcal{W}}^{1,p(.)}(\Omega)$ such that $\tilde{u}=0 \ p(.)\textnormal{-r.q.e on} \ \partial\Omega$. Then $\tilde{u}=0 \ p(.)\textnormal{-r.q.e on} \ \mathcal{B}$ and hence $\tilde{u}\in \mathcal{W}^{1,p(.)}_{0}(\Omega)$. Consequently
	\begin{eqnarray*}
		\left\lbrace u\in \tilde{\mathcal{W}}^{1,p(.)}(\Omega):\ \tilde{u}=0 \ p(.)\textnormal{-r.q.e. on} \ \partial\Omega \right\rbrace \subset \mathcal{W}^{1,p(.)}_{0}(\Omega). 	
	\end{eqnarray*} 
To prove the converse, observe that the set 	\begin{eqnarray*}
	\left\lbrace u\in \tilde{\mathcal{W}}^{1,p(.)}(\Omega):\ \tilde{u}=0 \ p(.)\textnormal{-r.q.e. on} \ \partial\Omega \right\rbrace	
\end{eqnarray*}
is a closed ideal of $\tilde{\mathcal{W}}^{1,p(.)}(\Omega)$ containing 
$ \mathscr{C}^{\infty}_{0}(\Omega)$. Hence it also contains 
$\mathcal{W}^{1,p(.)}_{0}(\Omega)$.      
\end{itemize}    
\end{proof}
The following corollary gives a necessary and sufficient condition in term of the relative $p(.)$-capacity for the equality $\mathcal{W}^{1,p(.)} _{0}(\Omega)=\tilde{\mathcal{W}}^{1,p(.)}(\Omega)$.
\begin{corollary}
Let $\Omega\subset \mathbb{R}^{n}$ be open. Then the following assertions are equivalent.
\begin{enumerate}
	\item $ C_{p(.)}^{\overline{\Omega}}(\partial\Omega)=0$;
	\item $\mathcal{W}^{1,p(.)} _{0}(\Omega)=\tilde{\mathcal{W}}^{1,p(.)}(\Omega)$.
\end{enumerate}
\end{corollary}
The next theorem present another description of the space $\mathcal{W}^{1,p(.)} _{0}(\Omega)$.
\begin{theorem}
	Let $\Omega\subset \mathbb{R}^{n}$ be open. Suppose that 
	$F$ is a closed subset of $\mathbb{R}^{n}$ such that $\Omega=\mathbb{R}^{n}\setminus F$. If $ F $ is of $p(.) $-Sobolev capacity zero, then  
		\begin{eqnarray*}
			\mathcal{W}^{1,p(.)}_{0}(\Omega)=\left\lbrace u_{\arrowvert\Omega}: \  u\in \mathcal{W}^{1,p(.)}(\mathbb{R}^{n})\right\rbrace. 	
		\end{eqnarray*} 
\end{theorem}
\begin{proof}
According to \cite{Hasto3} we can identify $\mathcal{W}^{1,p(.)} _{0}(\Omega)$ with a subspace of $\mathcal{W}^{1,p(.)}(\mathbb{R}^{n})$ in the following way:  
\begin{eqnarray*}
	\mathcal{W}^{1,p(.)}_{0}(\Omega)=\left\lbrace \tilde{u}\in \mathcal{W}^{1,p(.)}(\mathbb{R}^{n}):\ \tilde{u}=0 \ p(.)\textnormal{-q.e. in} \ \mathbb{R}^{n}\setminus \Omega  \right\rbrace, 	
\end{eqnarray*} 
where $\tilde{u}$ denotes the $p(.)$-quasi continuous representative of $u$. Since 
\begin{eqnarray*}
C_{p(.)}^{\overline{\Omega}}(\partial\Omega)=C_{p(.)}^{\overline{\Omega}}(F)\leqslant C_{p(.)}(F)=0,	
\end{eqnarray*} 
we deduce that
\begin{eqnarray*}
	\mathcal{W}^{1,p(.)}_{0}(\Omega)=\left\lbrace u_{\arrowvert\Omega}: \  u\in \mathcal{W}^{1,p(.)}(\mathbb{R}^{n})\right\rbrace. 	
\end{eqnarray*}   
\end{proof}
In the sequel, we prove that every function
$ u\in\mathcal{W}^{1,p(.)}_{0}(\Omega)\cap \mathscr{C}(\overline{\Omega})$ is zero everywhere on the boundary $\partial\Omega$ if and only if $\Omega$ is $p(.)$-regular in capacity.
\begin{definition}
	Let $\Omega\subset \mathbb{R}^{n}$ be an open set. We say that $\Omega$ is $p(.)$-regular in capacity if $ C_{p(.)}(B(x,r)\setminus\Omega)>0$ for all $ x\in\partial\Omega $.	
\end{definition}	
\begin{definition}
	We say that an open set $\Omega\subset \mathbb{R}^{n}$ has the $p(.)$-zero property if $u\in C(\partial\Omega)$ and $u=0\ p(.)$-q.e on $\partial\Omega$ implies that $u(x)=0$ for every 
	$x\in \partial\Omega$.
\end{definition}
\begin{theorem}\label{equivalece p capacity}
Let $ \Omega\subset \mathbb{R}^{n}$ be an open set. Then the following assertions are equivalent.
\begin{enumerate}
	\item $\Omega$ has the $p(.)$-zero property;
	\item  Every function $ u\in\mathcal{W}^{1,p(.)}_{0}(\Omega)\cap \mathscr{C}(\overline{\Omega})$ is zero everywhere on the boundary $\partial\Omega$;
	\item $\Omega$ is $p(.)$-regular in capacity;
	\item $ C_{p(.)}(B(x,r)\cap\partial\Omega)>0$ for all 
	$x\in \partial\Omega$;
	\item $ C_{p(.)}^{\overline{\Omega}}(B(x,r)\cap\partial\Omega)>0$ for all 
	$x\in \partial\Omega$;
	\item $u\in \mathscr{C}(\partial\Omega)$ and $u=0$ $p(.)$-r.q.e. on $ \partial\Omega $ implies that 
	$u(x)=0$ for every $x\in \partial\Omega$.  	 
\end{enumerate}	 
\end{theorem}
\begin{proof}
Before proving the theorem, we recall that from Theorem \ref{description of W^{1,p}_{0} by relative capacity} we have the following characterization of $ \mathcal{W}^{1,p(.)}_{0}(\Omega)$:
\begin{eqnarray*}
	\mathcal{W}^{1,p(.)}_{0}(\Omega)=\left\lbrace u\in \tilde{\mathcal{W}}^{1,p(.)}(\Omega):\ \tilde{u}=0 \ p(.)\textnormal{-r.q.e. on} \ \partial\Omega   \right\rbrace, 	
\end{eqnarray*} 
where $\tilde{u}$ denotes the relatively $p(.)$-quasicontinuous representative of $u$.\\ 
$ (1)\Longrightarrow (2) $: This follows from the fact that every function $ u\in\mathcal{W}^{1,p(.)}_{0}(\Omega)\cap \mathscr{C}(\overline{\Omega})$ is continuous and is zero everywhere on the boundary $\partial\Omega$.\\ 
$(2)\Longrightarrow (3)$: Assume that $\Omega$ is not $p(.)$-regular in capacity. Then there exist $x_{0}\in \partial\Omega$ such that $ C_{p(.)}(B(x,r)\cap\partial\Omega)=0$. Since 
\begin{eqnarray*}
C_{p(.)}^{\overline{\Omega}}(B(x_{0},r)\setminus\Omega)\leqslant C_{p(.)}(B(x_{0},r)\setminus\Omega),		
\end{eqnarray*}
we get that
\begin{eqnarray*}
	C_{p(.)}^{\overline{\Omega}}(B(x_{0},r)\cap\partial\Omega)=0.		
\end{eqnarray*}
Let $ \varphi\in \mathscr{C}^{\infty}_{0}(B(x_{0},r))$ be such that $ \varphi =1 $ on $B(x_{0},\frac{r}{2})$. Then $\varphi_{{|}_{\Omega}}\in \mathcal{W}^{1,p(.)}_{0}(\Omega)\cap \mathscr{C}(\overline{\Omega})$, but $ \varphi(x_{0})\neq0 $.\\ 
$(3)\Longrightarrow (4)$: Let  $x_{0}\in \partial\Omega$ such that $ C_{p(.)}(B(x,r)\cap\partial\Omega)=0$. We have that 
\begin{eqnarray*}
	B(x_{0},r)\setminus\Omega=(B(x_{0},r)\setminus\overline{\Omega})\cup((B(x_{0},r)\cap \partial\Omega)). 		
\end{eqnarray*} 
By the monotonicity of the $p(.)$-Sobolev capacity, we get that
\begin{eqnarray*}
	C_{p(.)}(B(x_{0},r)\setminus\Omega)&\leqslant& C_{p(.)}(B(x_{0},r)\setminus\overline{\Omega})+C_{p(.)}(B(x_{0},r)\cap \partial\Omega)\\
	&\leqslant& C_{p(.)}(B(x_{0},r)\setminus\overline{\Omega}).  		
\end{eqnarray*}
Since $ \Omega $ is $p(.)$-regular in capacity, we get that
$ C_{p(.)}(B(x_{0},r)\setminus\Omega)>0$.
It follows that $C_{p(.)}(B(x_{0},r)\setminus\overline{\Omega})>0$. Hence
$ B(x_{0},r)\setminus\overline{\Omega}\neq\varnothing$. Consequently there exist $x_{1}$ and $ \varepsilon_{1}>0$ such that $\overline{B(x_{1},\varepsilon_{1})}\subset B(x_{0},r)\setminus\overline{\Omega}$. Moreover there exist $x_{2}$ and $ \varepsilon_{2}>0$ such that $\overline{B(x_{2},\varepsilon_{2})}\subset B(x_{0},r)\cap\Omega$.\\
Let $\varepsilon\in(0,r)$ be such that
$ B(x_{1},\varepsilon_{1})\cup B(x_{2},\varepsilon_{2})\subset B(x_{0},\varepsilon)$. Let $u\in \mathscr{C}^{\infty}_{0}(B(x_{0},r))$ be such that $ u=1 $ on $B(x_{0},\varepsilon)$. Define the function $ v $ by 
	\begin{equation*}	
v:=
\left\{
\begin{array}{rl}
u & \text{on}\ B(x_{0},r)\cap\Omega; \\
0 & \text{on}\ \Omega^{c}. 
\end{array}
\right.
\end{equation*} 		   
Then $ v\in  \mathcal{W}^{1,p(.)}_{0}(\Omega)$ and $\nabla v=0$ a.e. on $B(x_{0},\varepsilon)$. Consequently $ v=C $ a.e. on $B(x_{0},\varepsilon)$. This contradicts the fact that $ v=1 $ on $ B(x_{0},\varepsilon_{2})$ and $ v=0$ on $ B(x_{0},\varepsilon_{1})$.\\ 
$(4)\Longrightarrow (5)$: Suppose that there exist $ x_{0}\in \partial\Omega$ such that \begin{equation*}
C_{p(.)}^{\overline{\Omega}}(B(x_{0},r)\cap \partial\Omega)=0. 
\end{equation*}
Let $ u\in \mathscr{C}^{\infty}_{0}(B(x_{0},r)) $ be such that $ u=1 $ on $ B(x_{0},\frac{r}{2})$. Then $u_{{|}_{\Omega}}\in \mathcal{W}^{1,p(.)}_{0}(\Omega)$. We define the function $ v $ by 
\begin{equation*}	
v:=
\left\{
\begin{array}{rl}
u & \text{on}\ \Omega; \\
0 & \text{on}\ \Omega^{c}. 
\end{array}
\right.
\end{equation*}
Then $ v\in \mathcal{W}^{1,p(.)}_{0}(\Omega)\subset \mathcal{W}^{1,p(.)}(\Omega)$ and $\nabla v=0$ a.e. on $B(x_{0},\frac{r}{2})$. Consequently $ v=1 $ a.e. on $B(x_{0},\frac{r}{2})$. This implies that $\tilde{v}=1 \ p(.) \textnormal{-q.e. on} \ B(x_{0},\frac{r}{2})$, where $ \tilde{v} $ denotes the $ p(.)$-quasicontinuous representative of $v$. In particular, $\tilde{v}=1 \ p(.) \textnormal{-q.e. on} \ B(x_{0},\frac{r}{2})\cap\partial\Omega$ and $\tilde{v}=0 \ p(.) \textnormal{-q.e.}$ on $ \mathbb{R}^{n} \setminus \Omega$. It follows that $ C_{p(.)}(B(x_{0},\frac{r}{2})\cap \partial\Omega)=0$.\\
$(5)\Longrightarrow (6)$: We suppose that there exist $ x_{0}\in\partial\Omega$ such that $ u(x_{0})\neq 0$. Since $ u\in \mathscr{C}(\partial\Omega)$, we deduce that there exist an open ball 
$ B(x_{0},r)$ such that $ u(x)\neq 0$ for all $ x\in  B(x_{0},r)\cap \partial\Omega $. Using the fact that $u=0 \ p(.) \textnormal{-q.e. on} \ \partial\Omega$ it follow that $ C_{p(.)}^{\overline{\Omega}}(B(x_{0},r)\cap \partial\Omega)=0$.\\
$(6)\Longrightarrow (1)$: Let $ u\in \mathscr{C}(\partial\Omega)$ be such that $ u=0 \ p(.) $-q.e. on $\partial\Omega$. By assertion 
$(6)$, we deduce that $ u(x)=0 $ for every 
$x\in\partial\Omega$.   		                
\end{proof}
\section{\textbf{Dirichlet problem in  nonlinear harmonic space} }

Let $\Omega\subset \mathbb{R}^{n}$ be a bounded open set and $p\in\mathscr{P}(\Omega)$ with $ 1<p^{-}\leqslant p(x)\leqslant p^{+}<n$. 
We will assume that the variable exponent $ p $ satisfies the logarithmic H\"{o}lder continuity condition introduced by Zhikov in \cite{Zhikov}, namely we suppose that
\begin{equation*}
\left|p(x)-p(y) \right|\leq \dfrac{C}{-\log(\left|x-y \right| )} \ \text{for} \ (x,y)\in\Omega\times\Omega \ \text{with} \ \left|p(x)-p(y) \right|<\dfrac{1}{2},  
\end{equation*}
where $ C > 0 $ is independent of $ (x, y) $. Under this assumption and as it was observed in \cite{Zhikov}, the space of smooth functions is dense in the variable exponent Sobolev space. We refer also to the monograph \cite{Hasto} for the details in this direction.

For $f\in \mathscr{C}(\partial\Omega)$, we consider the Dirichlet problem:
\begin{equation*}	
	\left\{
	\begin{array}{rl}	
		 \mathbf{\mathscr{L}_{p(.)}}u:=-\Delta_{p(.)}u+\mathscr{B}(.,u)=0 & \text{in}\ \Omega; \\
		u=f & \text{on}\ \partial\Omega, 
	\end{array}
	\right.
\end{equation*}
where $\Delta_{p(.)}:=\div\left( \left| \nabla u \right|^{p(.)-2} \nabla u \right)$ is the $p(.)$-Laplacian operator and  $\mathscr{B}:\Omega
\times \mathbb{R}\rightarrow \mathbb{R}$ is a given Carath\'{e}odory
functions satisfies the following structural condition for a positive constant $ C $:
\begin{itemize}
	\item $ \left| \mathscr{B}(.,\zeta)  \right|\leqslant a(.)+C\left|\zeta\right|^{p(.)-1}$ a.e. in $ \Omega $ and for all $\zeta \in \mathbb{R}^{n}$, where $ a $ is a positive measurable function lying in 
	$ L^{p^{'}}(\Omega)$.
	\item $ \zeta\mapsto \mathscr{B}(.,\zeta)  $ is increasing function on $ \Omega $.  
\end{itemize}
A typical example when the above structural condition hold is $ \mathscr{B}(.,u) = \left|u \right|^{p(.)-2}u$.

In the sequel, we discuss the nonlinear potential theory associated with the 
$\mathcal{H}_{\mathscr{L}_{p(.)}}$-harmonic sheaf defined as follows:
\begin{equation*}
\mathcal{H}_{\mathscr{L}_{p(.)}}(\Omega)=
\left\lbrace u\in \mathscr{C}(\Omega):\mathscr{L}_{p(.)}u=0 \right\rbrace. 
\end{equation*}
Element in the set $ \mathcal{H}_{\mathscr{L}_{p(.)}}(\Omega)$ are called $\mathscr{L}_{p(.)}$-harmonic on $ \Omega $. First, we give some topics from the theory of abstract nonlinear harmonic spaces.\\
Let $ (X, \mathcal{T} )$ be a topological Hausdorff space, locally connected and locally compact. We also assume that the topology $ \mathcal{T} $ has a countable basis.  Let $\mathcal{H}$ be a sheaf of continuous real-valued functions on $ X $. The sheaf $\mathcal{H}$ is usually called a harmonic sheaf and its functions harmonic functions.
\begin{definition}
 We say that an open set $ \mathscr{V} \in \mathcal{T}$ is $\mathcal{H}$-regular if the following conditions are satisfied:\\
($ C_{1} $): $\overline{\mathscr{V}}$ is compact and $\partial \mathscr{V}\neq \varnothing$;\\
($ C_{2} $): for every continuous function $ f : \partial \mathscr{V} \longrightarrow \mathbb{R}$, there exists a unique $\mathcal{H}$-harmonic function in $\mathscr{V}$, denoted by $H^{\mathscr{V}}_{f}$, such that
	\begin{equation*}
		\lim_{x\longrightarrow y} H^{\mathscr{V}}_{f}(x)=f(y) \ \ \text{for every} \ y\in \partial \mathscr{V};  
	\end{equation*}
($ C_{3} $): For each pair $  f, g \in \mathscr{C}(\partial U)$ the condition $ f \leqslant g $ implies $ H^{\mathscr{V}}_{f} \leqslant H^{\mathscr{V}}_{g} $ in $ U $.
\end{definition}
\begin{definition}
The pair $ (X, \mathcal{H}) $ is called a nonlinear harmonic space, with respect to the sheaf $\mathcal{H}$ if the following three axioms are satisfied:\\
($ A_{1} $) The family of the $\mathcal{H}$-regular open sets is a basis for the topology of $ X$.\\
($ A_{2} $) the sheaf $\mathcal{H}$ is non degenerate: For every $ x\in X $  there exists an open neighborhood $ V $ of $ x $ and a function $ h \in \mathcal{H}(V)$ such that $ h(x) \neq 0$.\\
($ A_{3} $) Convergence property: 
for every subset $ U $ of $ X $ and every monotone locally bounded sequence $ (h_{n})_{n} $ in $\mathcal{H}(U)$ , we have $ h= \lim_{n\longrightarrow +\infty}h_{n}\in \mathcal{H}(U) $.
\end{definition}
Using the notion of traces of functions and uniform approximation in $\mathcal{W}^{1,p(.)}(\Omega)$, Baalal
and Berghout \cite{BA3} showed the following theorem:
\begin{theorem}\label{Direchle solution}
	Let $f\in \mathscr{C}(\partial\Omega)$. Then there exists a unique continuous $\mathscr{L}_{p(.)}$-harmonic extension of $f$ in $\overline{\Omega}$. 
\end{theorem}
Note that the comparison principle given in \cite{BA2} can be extended immediately to functions in $ \mathscr{C}(\partial \Omega)$ in the following way:
\begin{proposition}(Comparison principal).\label{Comparison principal}
Let $  f, g \in \mathscr{C}(\partial \Omega)$. If $ f \leqslant g $ on $ \partial \Omega $, then $\mathscr{L}^{\Omega}_{p(.),f}  \leqslant \mathscr{L}^{\Omega}_{p(.),g}$ in $ \Omega $.
\end{proposition}
Combining Theorem \ref{Direchle solution} with Proposition \ref{Comparison principal}, we obtain the following corollary.
\begin{corollary}\label{basis}
Let $\Omega\subset \mathbb{R}^{n}$ be a bounded open set, then 
$\Omega $ is $\mathcal{H}_{\mathscr{L}_{p(.)}}$-regular. Moreover, the family of the $\mathcal{H}_{\mathscr{L}_{p(.)}}$-regular open sets is a basis for the topology of $ \Omega$.  
\end{corollary}
\begin{theorem}
Let $ \mathcal{H}_{\mathscr{L}_{p(.)}} $ be  the sheaf of $\mathscr{L}_{p(.)}$-harmonic functions. Then the pair $ (\Omega,\mathcal{H}_{\mathscr{L}_{p(.)}} )$ is a nonlinear harmonic space.
\end{theorem}
\begin{proof}
By Corollary \ref{basis}, the family of the $\mathcal{H}_{\mathscr{L}_{p(.)}}$-regular open sets is a basis for the topology of $ \Omega$. From \cite{BA5}, we deduce  the  following form of the Harnack inequality: for every non empty open set $ \Omega \subset \mathbb{R}^{n}$, for every constant $ M > 0  $ and every compact $K\subset\Omega $, there exists a constants $ C_{1} $ and $ C_{2} $ such hat
\begin{equation*}
\text{esssup}_{x\in K}u(x)\leqslant C_{1} (\text{essinf}_{x\in K}u(x)+C_{2}),
\end{equation*}
for every positive $ u $ in $ \mathcal{H}_{\mathscr{L}_{p(.)}}(\Omega) $ with $  u \leqslant M $. It follows that the sheaf $ \mathcal{H}_{\mathscr{L}_{p(.)}}(\Omega) $ is non degenerate.\\
Let $ (u_{n})_{n} $ be a locally bounded increasing sequence in $ \mathcal{H}_{\mathscr{L}_{p(.)}}(\Omega) $. Then by Harnack inequality the family 
\begin{equation*}
\mathcal{F}:=\left\lbrace u_{n}(x), x\in \overline{V},n\in \mathbb{N}, \ \text{for every } \ V\subset \overline{V} \subset \Omega   \right\rbrace 
\end{equation*}  
is uniformly equicontinuous on $\overline{V} $ and it is a bounded subset of $ \mathscr{C}(\overline{V})$. By Ascoli–Arzelà theorem $\mathcal{F} $ has compact closure in $ \mathscr{C}(\overline{V})$; consequently the sequence $ (u_{n} )_{n}$ converges locally and uniformly in $ \Omega $ to a continuous function $ u $. Hence, for every $ \varepsilon >0 $, there exists $ n_{0}\in \mathbb{N}$ such that:
\begin{equation*}
(\forall n\geq n_{0}); \sup_{x\in \overline{V}}\left|u_{n}(x)-u(x) \right|\leq \varepsilon.  
\end{equation*}    
By the comparison principal
 \begin{equation*}
 \mathscr{L}_{p(.)}u-\varepsilon\leq u_{n}\leq \mathscr{L}_{p(.)}u+\varepsilon,
 \end{equation*}    
letting $ n\longrightarrow +\infty $, we get 
\begin{equation*}
\mathscr{L}_{p(.)}u-\varepsilon\leq u\leq \mathscr{L}_{p(.)}u+\varepsilon.
\end{equation*}
Now, letting $ \varepsilon\longrightarrow 0$, we get $\mathscr{L}_{p(.)}u=u$. Hence the space $ (\Omega,\mathcal{H}_{\mathscr{L}_{p(.)}} )$ satisfies the convergence property.    
\end{proof}
Because of the nonlinearity of the sheaf 
$\mathcal{H}_{\mathscr{L}_{p(.)}}$ there is no standard connection between the Dirichlet problem and harmonic measure as in linear axiomatic potential theories.\\    
Notice that by the uniqueness of the solution, we deduce that the $\mathscr{L}_{p(.)}$-harmonic function and the solution given by Brelot-Perron-Wiener method coincide. In this case the space of continuous functions on $\partial\Omega$ is resolutive. Hence, in a nonlinear harmonic space $ (\Omega,\mathcal{H}_{\mathscr{L}_{p(.)}} )$ we construct the Perron–Wiener– Brelot operator
\begin{equation*}
f\mapsto \mathscr{L}^{\Omega}_{p(.),f} 	
\end{equation*}
from a set of real extended functions $ f : \partial\Omega \longrightarrow \mathbb{R}$ to the nonlinear harmonic space $ (\Omega,\mathcal{H}_{\mathscr{L}_{p(.)}} )$ of the $ \mathscr{L}_{p(.)}$-harmonic functions in $ \Omega $.
  
The following theorem suggests that the mapping $ f\mapsto {\mathscr{L}^{\Omega}_{p(.),f}}$ from $ \mathscr{C}(\partial\Omega) $ into $ \mathcal{H}_{\mathscr{L}_{p(.)}} $ is injective whenever $ \Omega $ is $p(.)$-regular in capacity.  
\begin{theorem}
Let $\Omega\subset \mathbb{R}^{n}$ be a bounded open set. Assume that $ \Omega $ is $p(.)$-regular in capacity. Then the mapping $ f\mapsto \mathscr{L}^{\Omega}_{p(.),f}$ from $ \mathscr{C}(\partial\Omega) $ into $ \mathcal{H}_{\mathscr{L}_{p(.)}}(\Omega) $ is injective.
\end{theorem}
\begin{proof}
Let $ f,g\in \mathscr{C}(\partial\Omega)$ and assume that $ \mathscr{L}^{\Omega}_{p(.),f}=\mathscr{L}^{\Omega}_{p(.),g}$. By Theorem \ref{Direchle solution}, $ \mathscr{L}^{\Omega}_{p(.),f},\mathscr{L}^{\Omega}_{p(.),g} \in \mathscr{C}(\overline{\Omega})$. Hence 
\begin{equation*}
\mathscr{L}^{\Omega}_{p(.),f}=f \ p(.)\text{-r.q.e on} \ \partial\Omega   
\end{equation*}
and
\begin{equation*}
\mathscr{L}^{\Omega}_{p(.),g}=g  \ p(.)\text{-r.q.e on} \ \partial\Omega.  
\end{equation*}
Again by Theorem \ref{Direchle solution}, we deduce that 
\begin{equation*}
f-g=0  \ p(.)\text{-r.q.e on} \ \partial\Omega.  
\end{equation*}
Since $ \Omega $ is $p(.)$-regular in capacity, it follows from Theorem \ref{equivalece p capacity} that $ f-g=0$ everywhere on $ \partial\Omega $,  
which implies that $ f=g $  everywhere on $ \partial\Omega $. Hence the mapping $ f\mapsto \mathscr{L}^{\Omega}_{p(.),f}$ from $ \mathscr{C}(\partial\Omega) $ into $ \mathcal{H}_{\mathscr{L}_{p(.)}}(\Omega) $ is injective.       
\end{proof}
\begin{corollary}
Let $\Omega\subset \mathbb{R}^{n}$ be a bounded open set. Let $\mathscr{R}(\partial\Omega)$ the class of all real valued resolutive function on 
$\partial\Omega$ and $f,g\in  \mathscr{R}(\partial\Omega)$. Suppose that $ \Omega $ is $p(.)$-regular in capacity. If 
$\mathscr{L}^{\Omega}_{p(.),f}=\mathscr{L}^{\Omega}_{p(.),g}$, then the set $ \left\lbrace x\in \partial \Omega: f(x)\neq g(x)\right\rbrace$ is negligible.   
\end{corollary}
 

\bibliographystyle{amsplain}

\end{document}